\documentclass[a4paper,11pt]{amsart}

\usepackage[english]{babel}
\usepackage[centertags]{amsmath}
\usepackage{latexsym, amsfonts, amssymb, graphicx}

\usepackage[all,dvips]{xy}
\usepackage[latin1]{inputenc}
\usepackage{url}
\usepackage{enumerate}
\usepackage{multirow}

\theoremstyle{plain}
\newtheorem{thm}{Theorem}[section]
\newtheorem{lem}[thm]{Lemma}
\newtheorem{prop}[thm]{Proposition}
\newtheorem{cor}[thm]{Corollary}
\newtheorem{question}{Question}

\theoremstyle{definition}
\newtheorem{defn}[thm]{Definition}
\newtheorem{nota}[thm]{Notation}

\theoremstyle{remark}
\newtheorem{rem}[thm]{Remark}

\newtheorem{exmp}[thm]{Example}

\def\N{{\mathbf N}}

\def\P{{\mathbf P}}

\def\L{{\mathcal L}}

\newcommand{\F}{\mathbf{F}}
\newcommand{\res}{\textrm{res}}
\newcommand{\supp}{\textrm{Supp}}

\begin{document}

\title[Differential Approach for Duals of AG codes on Surfaces]{Differential Approach for the Study of Duals of Algebraic-Geometric Codes on Surfaces}
\author{\sc Alain Couvreur}

\address{Alain Couvreur\\
INRIA Saclay, Projet \textsc{Tanc}\\
École polytechnique\\
Laboratoire d'informatique LIX, UMR 7161\\
91128 Palaiseau Cedex,
France}
\email{alain.couvreur@inria.fr}

\maketitle

\thispagestyle{empty}

\begin{abstract}
The purpose of the present article is the study of duals of functional codes on algebraic surfaces.
We give a direct geometrical description of them, using differentials.
Even if this geometrical description is less trivial, it can be regarded as a na\-tural extension to surfaces of the result asserting that the dual of a functional code $C_L (D,G)$ on a curve is the differential code $C_{\Omega}(D,G)$ .
We study the parameters of such codes and state a lower bound for their minimum distance.
Using this bound, one can study some examples of codes on surfaces, and in particular surfaces with Picard number $1$ like elliptic quadrics or some particular cubic surfaces. 
The parameters of some of the studied codes reach those of the best known codes up to now.
\end{abstract}

\section*{Introduction}
Given a variety $X$ over a finite field, a divisor $G$ on $X$ and a family $P_1, \ldots , P_n$ of rational points of $X$, one can construct the \textit{functional code} $C_L (X, \Delta, G)$, where $\Delta$ denotes the formal sum $P_1+ \cdots +P_n$.
This construction, due to Manin in \cite{manin}, is obtained by evaluating the global sections of the sheaf $\L (G)$ at the points $P_1, \ldots , P_n$.
Basically, the aim of this paper is to get information on the dual $C_L (X, \Delta, G)^{\bot}$ of such a functional code.

Most of the literature on algebraic--geometric codes deals with the case when $X$ is a curve.
In this situation, the dual code $C_L(X, \Delta, G)^{\bot}$ is equal to the \textit{differential} code $C_{\Omega}(X, \Delta, G)$ whose construction, due to Goppa in \cite{goppa}, involves residues of differential forms on $X$.
Moreover, on curves, it is also well-known that a differential code $C_{\Omega}(X, \Delta, G)$ is equal to a functional code $C_L (X, \Delta, G')$, where $G'$ is a divisor depending on $G, \Delta$ and the canonical class of $X$.
Therefore,  the study of duals of functional codes on curves is equivalent to the study of functional codes.

For higher--dimensional varieties, the geometric problems raised by coding theory become much more difficult and hence only little is known.
Most of the literature on the topic concerns the estimation of the parameters and in particular the minimum distance of functional codes on particular surfaces.
For instance, codes on quadric varieties are studied in \cite{aubryquad} and \cite{fred}, codes on surfaces with Picard number $1$ are studied in \cite{zarzar}
(see the survey chapter of J.B. Little in \cite{bouquindiego} for a detailed survey on the topic).
Concerning the dual of such a functional code, almost nothing is known.
In \cite{oim}, a differential construction for codes on surfaces is given, which turns out to be a natural extension to surfaces of Goppa's construction on curves (see \cite{goppa}).
It is proved in the same article that such a differential code is contained in the dual of a functional code, but that the converse inclusion is false in general.

The aim of the present paper is to get general information on duals of functional codes on surfaces.
For that, we try to answer two questions asked in Section \ref{stat}.
The first one (which was actually raised in the end of \cite{oim}) is to find a direct geometrical description of such a code using differentials.
The second one is to get information on the parameters of such codes.
As an answer for the first question, we state and prove Theorem \ref{1real}. This statement asserts that even if the dual of a functional code on a surface is not differential in general, it is always a sum of differential codes on this surface.
Afterwards, we focus our study on the estimate of the parameters of such a code and state results yielding a lower bound for its minimum distance.
When the surface is the projective plane, these results yield the exact minimum distance which is already known in this case since the codes are Reed--Muller (see \cite{delsarte} Theorem 2.6.1).
In addition, these results (Theorems \ref{Bound} and \ref{impr}) are easy to handle provided the Picard number of the surface is small. 
It is worth noting that the works on parameters of codes on surfaces point out that surfaces with Picard number $1$ yield good functional codes.
This principle was first observed by Zarzar in \cite{zarzar} and is confirmed by some other works on the topic.
For instance, one sees in \cite{fred} that elliptic quadrics (which have Picard number $1$) give much better codes than hyperbolic ones (which have Picard number $2$).
It turns out that this principle asserting that surfaces with small Picard number yield good functional codes seems to hold for duals of functional codes. Two examples of surfaces with Picard number $1$ are studied (namely, elliptic quadrics and cubic which do not contain rational lines). The minimum distance of some dual codes obtained from these examples turn out to reach the best known minimum distance up to now compared to their length and dimension. 

\subsection*{Contents}
Notations are given in Section \ref{context}. They are followed by the recall of some prerequisites in Section \ref{prere}.
The aims of the present article are summarised in Section \ref{stat}, where Questions \ref{qrealsum} and \ref{howtobound} are raised.
Section \ref{funcconst} is devoted to the proof of some statements which are important in what follows.
In particular, Proposition \ref{keytool}, which is the key tool for the proof of the two main results (Theorems \ref{1real} and \ref{Bound}), is proved in this section.
Section \ref{secreal} is devoted to the answer to Question \ref{qrealsum}.
Theorem \ref{1real} is proved in this section and asserts that, even if the dual of a functional code on a surface is not in general a differential code on this surface, it is always a sum of differential codes on this surface.
Section \ref{secmind} is devoted to the answer to Question \ref{howtobound}, that is the study of the minimum distance of the dual of a functional code on a surface.
Two results are stated: Theorem \ref{Bound}, yielding a lower bound for the minimum distance of some of these codes, and Theorem \ref{impr}, which improves the bound given by Theorem \ref{Bound} in some situations.
Some applications of Theorems \ref{Bound} and \ref{impr} are studied in Section \ref{ex1}, and lower bounds for the minimum distance are given for explicit examples. The parameters of these codes are compared with those of the best known codes up to now (found in \cite{codetables} and \cite{minT}).

\section{Notations}\label{context}

\subsection{About coding theory}
An error--correcting code is a vector subspace $C$ of $\F_q^n$ for some positive integer $n$.
The integer $n$ is called the \textit{length} of $C$.
Elements of $C$ are called \textit{codewords}.
The \textit{Hamming weight} $w(c)$ of a vector $c\in \F_q^n$ is the number of its nonzero coordinates. 
The \textit{Hamming distance} $d(x,y)$ between two vectors $x,y \in \F_q^n$ is $d(x,y):=w(x-y)$.
Given a code $C \in \F_q^n$, the \textit{minimum distance} $d$ of $C$ is the smallest Hamming distance between two distinct elements of $C$.
A code is said to have para\-meters $[n,k,d]$ if its length is $n$, its dimension over $\F_q$ is $k$ and its minimum distance is $d$.

On $\F_q^n$, we consider the canonical pairing $\langle . ,. \rangle$ defined by $\langle x,y \rangle:=\sum_{i=1}^n x_iy_i$.
Given a code $C \subset \F_q^n$, its orthogonal space $C^{\bot}$ for this pairing is called \textit{dual code} of $C$.

\subsection{About divisors and sheaves}
Given a sheaf $\mathcal{F}$ on a variety $X$, we denote by $\mathcal{F}_P$ its stalk at a point $P\in X$.
Linear equivalence between divisors is denoted by $D \sim D'$.
Given a map $\nu: Y\hookrightarrow X$ between two varieties and a divisor $G$ on $X$, then, for convenience's sake, the pullback $\nu^{\star}G$ is denoted by $G^{\star}$ whenever there is no possible confusion on $\nu$.
Given a projective variety $V$, we denote by $H_V$ the hyperplane section of $V$ and by $K_V$ its canonical class.

\subsection{About intersections}
Let $S$ be an algebraic surface, $P$ be a smooth point of $S$ and $X,Y$ be two curves embedded in $S$.
If $X$ and $Y$ have no common irreducible component in a neighbourhood of $P$, we denote by $m_P(X,Y)$ the intersection multiplicity of $X$ and $Y$ at $P$.
The notion of intersection multiplicity extends by linearity to divisors on $S$.
Finally, the intersection product of two divisor classes $D$ and $D'$ is denoted by $D.D'$.

\subsection{Base field extensions} Let $X$ be a variety defined over $\F_q$. We denote by $\overline{X}$ the variety $\overline{X}:=X\times_{\F_q} \overline{\F}_q$. 
In the same way, let $\mathcal{F}$ be a sheaf on $X$, then we denote by $\overline{\mathcal{F}}$ the pullback of $\mathcal{F}$ on $\overline{X}$.

\section{Prerequisites}\label{prere}

In this section we recall some facts about residues and differential forms on surfaces.
Afterwards, we give some necessary prerequisites on algebraic--geometric codes on surfaces.

\subsection{Residues of differential $2$-forms on algebraic surfaces}\label{ressurf}
For further details on the definitions and the statements given in the present subsection, see \cite{oim} and \cite{thesebibi}. Some results on residues can also be found in \cite{parshin}.

\subsubsection{Residues in codimension $1$}\label{1res}
Let $C$ be an irreducible curve embedded in a smooth surface $S$ over an arbitrary field $k$.
If $\omega$ is a differential $2$--form on $S$ with valuation $\geq -1$ along $C$, then one can define a $1$--form on $C$ denoted by $\res^1_{C}(\omega)$.
See \cite{oim} Definition 1.3.

\subsubsection{Residues in codimension $2$}\label{subres}
Let $C$ be an irreducible curve embedded in a surface $S$ and $P$ be a rational point of $S$.
Given a $2$-form $\omega$ on $S$, one defines a residue at $P$ along $C$ of $\omega$ denoted by $\res^2_{C,P}(\omega)$ (see \cite{oim} Definition 3.1 and Theorem 3.6).
By convention, the map $\res^2_{C,P}$ is identically zero when $P\notin C$.
This notion generalises to any arbitrary reduced curve $C$.
In this situation, $\res^2_{C,P}(\omega)$ is the sum of the residues of $\omega$ at $P$ along each irreducible component of $C$.
Finally, if $D$ is a divisor on $S$, we denote by $\res^2_{D,P}$ the residue at $P$ along the \textbf{reduced support} of $D$.
That is
$
\res^2_{D,P}:=\res^2_{\supp (D),P}
$.
The following proposition summarises the properties of $2$--residues we need in what follows.

\begin{prop}\label{2res}
  Let $S$ be a smooth surface over an arbitrary field, $D$ be a divisor on $S$ and $P$ be a rational point of $S$.
Let $\omega$ be a rational $2$--form on $S$.
\begin{enumerate}[(i)]
\item\label{2res1} If in a neighbourhood of $P$, the pole locus of $\omega$ has no common component with $\supp (D)$, then $\res^2_{D,P}(\omega)=0$.
\item\label{2res2} If  in a neighbourhood of $P$, the pole locus of $\omega$ is entirely contained in $\supp (D)$, then $\res^2_{D,P}(\omega)=0$.
\end{enumerate}
In addition, let $C\subset S$ be a smooth curve at $P$.
\begin{enumerate}[(i)]
  \setcounter{enumi}{2}
\item\label{2res3}  If $\omega$ has valuation $\geq -1$ along $C$, then $\res^2_{C,P}(\omega)=\res_P (\res^1_C (\omega))$.
\end{enumerate}
\end{prop}

\begin{proof}
  The definition of $2$--residues (\cite{oim} Definition 3.1) gives (\ref{2res1}).
From \cite{oim} Theorem 6.3, we get (\ref{2res2}).
Finally (\ref{2res3}) is a consequence of \cite{oim} Definitions 1.4 and 3.1 together with Remark 3.3.
\end{proof}

\begin{rem}\label{machin1}
 Basically, Proposition \ref{2res} asserts that $\res^2_{D,P}(\omega)$ is nonzero if and only if in any neighbourhood of $P$, the support of $D$ contains at least one component of the pole locus of $\omega$ but does not contain entirely this pole locus.
It entails in particular that nonzero residues appear only at points $P$ at which two distinct poles of $\omega$ meet.
\end{rem}

\subsection{Algebraic--geometric codes on surfaces}

\subsubsection{Context}\label{contextcodes} Let $S$ be a smooth projective geometrically connected surface over a finite field $\F_q$, let $G$ be a divisor on $S$ and $P_1, \ldots, P_n$ be a family of rational points of $S$ avoiding the support of $G$. Denote by $\Delta$ the $0$--cycle $\Delta:=P_1+ \cdots +P_n$.

\subsubsection{Functional codes} Recall the definition, due to Manin in \cite{manin}, of the functional code associated to $G$ and $\Delta$. This code is defined to be the image of the evaluation map
$$
ev_{\Delta}:
\left\{
\begin{array}{ccc}
H^0(S, \L (G)) & \longrightarrow & \F_q^n \\
f & \longmapsto & (f(P_1), \ldots , f(P_n)).   
\end{array}
\right.
$$
It is denoted by $C_L (S,\Delta, G)$ or $C_L(\Delta, G)$ if there is no possible confusion on the involved variety. 

\subsubsection{Differential codes}\label{secDiff}
A differential construction of codes on surfaces is given in \cite{oim} 8.1.
Let $D_a, D_b$ be two divisors on $S$ whose supports have no common component, the differential code associated to $\Delta, D_a,D_b$ and $G$ is the image of the map
$$
res^2_{D_a,\Delta}:
\left\{
\begin{array}{ccc}
H^0(S, \Omega^2 (G-D_a-D_b)) & \longrightarrow & \F_q^n \\
\omega & \longmapsto & (\res^2_{D_a,P_1}(\omega), \ldots , \res^2_{D_a,P_n}(\omega)).   
\end{array}
\right.
$$ 
It is denoted by $C_{\Omega} (S,\Delta,D_a,D_b,G)$ or $C_{\Omega}(\Delta, D_a, D_b, G)$ when there is no possible confusion on the involved surface.

If there is no relation between the pair $(D_a, D_b)$ and $\Delta$, then there is no interesting relation between $C_L (S,\Delta,G)$ and $C_{\Omega} (S,\Delta,D_a,D_b,G)$.
This motivates the notion of $\Delta$--convenient pair of divisors.

\begin{defn}[$\Delta$--convenience, \cite{oim} Definition 8.3] 
A pair $(D_a, D_b)$ is said to be $\Delta$--convenient if
\begin{enumerate}
\item[$(i)$] the supports of $D_a$ and $D_b$ have no common irreducible component; 
\item[$(ii)$] for all $P\in \overline{S}$, the map $\res^2_{D_a,P}:\ \overline{\Omega^2(-D_a-D_b)}_P \rightarrow \overline{\F}_q$ is $\mathcal{O}_{\overline{S},P}$--linear;
\item[$(iii)$] this map is surjective for all $P\in \supp (\Delta)$ and zero elsewhere.  
\end{enumerate}
\end{defn}

\begin{rem}
Some examples and pictures illustrating this notion are given in \cite{thesebibi} II.3.4 and 5.
An explicit criterion for $\Delta$--convenience involving intersection multiplicities is given in \cite{oim} Proposition 8.6.
\end{rem}

\noindent In what follows, we also use a weaker definition called \textit{sub--$\Delta$--convenience}.

\begin{defn}[Sub--$\Delta$--convenience, \cite{thesebibi} III.2.1]\label{sub} A pair $(D_a,D_b)$ is said to be sub--$\Delta$--convenient if it is $\Delta'$--convenient for some $0\leq \Delta' \leq \Delta$. Equivalently, the pair satisfies the conditions $(i)$ and $(ii)$ of the previous definition together with 
  \begin{enumerate}
  \item[$(iii')$] for all $P\in \overline{S}\smallsetminus \supp (\Delta)$, the map $\res^2_{D_a, P}: \overline{\Omega^2(-D_a-D_b)}_P \rightarrow \overline{\F}_q$ is zero.
  \end{enumerate}
  \end{defn}




\section{Statement of the problems}\label{stat}

On a curve $X$ with a divisor $G$ and a sum of rational points $D$ (which is also a divisor), it is well-known that the dual of the functional code $C_L (X,D,G)$ equals the differential code $C_{\Omega} (X,D,G)$ (for instance see \cite{sti} II.2.8).
On a surface $S$ with a divisor $G$ and a sum of rational points $\Delta$ (which is not a divisor!), the situation is not that simple.
Nevertheless, it has been proved in \cite{oim} Theorem 9.1, that, if $(D_a,D_b)$ is a $\Delta$--convenient pair, then $C_{\Omega} (S,\Delta,D_a,D_b,G) \subseteq C_L (S,\Delta,G)^{\bot}$.

\begin{rem}\label{machin2}
This holds for a sub--$\Delta$--convenient pair (with the very same proof). 
\end{rem}

As said in the introduction, the reverse inclusion is in general false.
This motivates the following questions (the first one is raised in the end of \cite{oim}).

\begin{question}\label{qrealsum}
  Can the code $C_L (S,\Delta, G)^{\bot}$ be realised as a sum of differential codes on $S$ associated to different pairs of (sub--)$\Delta$--convenient divisors?
\end{question}

\medbreak 

\noindent \textbf{Question \ref{qrealsum}b.} \emph{Given $c\in C_L (S,\Delta, G)^{\bot}$, does there exist a (sub--)$\Delta$--convenient pair $(D_a,D_b)$ such that $c\in C_{\Omega} (S,\Delta, D_a, D_b, G)$?}

\begin{question}\label{howtobound}
  How can one estimate or find a lower bound for the minimum distance of the code $C_L(S,\Delta, G)^{\bot}$?
\end{question}

Theorem \ref{1real} answers positively to Question \ref{qrealsum}b, which entails a positive answer for Question \ref{qrealsum} (see Corollary \ref{totalreal}).
Theorems \ref{Bound} and \ref{impr} yield a method to estimate the minimum distance of duals of functional codes.

\begin{rem}
Actually, it is proved in \cite{thesebibi} \S III.3 that Questions \ref{qrealsum} and \ref{qrealsum}b are equivalent.
However, such a proof is not necessary in what follows.
\end{rem}

\section{The main tools}\label{funcconst}

The present section contains some tools which are needed to prove the main results of this article (Theorems \ref{1real} and \ref{Bound}).
In particular, Proposition \ref{keytool}, which is the key tool of this paper is proved here.

The reader interested in the results and their applications can skip this section in a first reading and look at the applications in Sections \ref{secreal} and \ref{secmind}.

\subsection{A problem of interpolation}
The proofs of Proposition \ref{keytool} and Theorem \ref{1real} need some result due to Poonen in \cite{poon} Theorem 1.2.
To state this result, we need to introduce some notations and definitions.

\begin{nota}
  For all integers $d,r \geq 0$, we denote by $S_{d,r}$ the subspace of $\F_q[X_0, \ldots , X_r]$ of homogeneous polynomials of degree $d$.
We denote then by $S_{r}$ the set $S_{r}:=\cup_{d\geq 0} S_{d,r}$.
\end{nota}

\begin{defn}[Poonen, \cite{poon} \S 1]
  The density $\mu (\mathcal{P})$ of a part $\mathcal{P}$ of $S_r$ is defined by
$$
\mu(\mathcal{P}):=\lim_{d\rightarrow +\infty} \frac{\sharp(\mathcal{P} \cap S_{d,r})}{\sharp S_{d,r}} \cdot
$$
\end{defn}

\begin{thm}[Poonen, \cite{poon} 1.2]\label{poonen} 
Let $X$ be a quasi-projective sub scheme of $\P^r$ over $\F_q$.
Let $Z$ be a finite sub-scheme of $\P^r$, and assume that $U:=X \setminus (Z\cap X)$ is smooth of dimension $m\geq 0$.
Fix a subset $T \subseteq H^0(Z, \mathcal{O}_Z)$.
Given $f\in S_{d,r}$, let $f_{|Z}$ be the element of $H^0(Z, \mathcal{O}_Z)$ that on each connected component $Z_i$ equals the restriction of $X_j^{-d}f$ to $Z_i$, where $j=j(i)$ equals the smallest $j\in \{0, \ldots , n\}$ such that the coordinate $X_j$ is invertible on $Z_i$.
Define 
$$
\mathcal{P}:=\{f\in S_r:\ \{f=0\}\cap U\ \textrm{is smooth of dimension}\ m-1,\ \textrm{and}\ f_{|Z}\in T\}.
$$
Then,
$$
\mu (\mathcal{P})=\frac{\sharp T}{\sharp H^0 (Z, \mathcal{O}_Z)}{\zeta_U (m+1)}^{-1},
$$
where $\zeta_U(s)=Z_U(q^{-s})$ denotes the Zeta function of $U$.
\end{thm}

\begin{cor}\label{utile}
  Let $S$ be a smooth projective surface over $\F_q$ and $Q_1, \ldots, Q_s$ be a finite set of rational points of $S$.
There exists  an integer $s\geq 0$ such that for all $d\geq s$, there exists a hypersurface $H$ of degree $d$ in $\P^r$ whose scheme--theoretic intersection with $S$ is smooth of codimension $1$
and contains $Q_1, \ldots, Q_s$.
\end{cor}

\begin{proof}
For $j\in \{1,\ldots ,s\}$ denote by $\mathcal{I}_{j}$ the sheaf of ideals of $\mathcal{O}_X$ corresponding to $Q_j$.
Let $\mathcal{I}$ be the sheaf of ideals $\mathcal{I}:=\mathcal{I}_{1}\cdots \mathcal{I}_{s}$.
Denote by $Z$ the non-reduced sub-scheme of $X$ defined by the finite set $\{Q_1, \ldots , Q_s\}$ with the structure sheaf $\mathcal{O}_Z:=\mathcal{O}_S / \mathcal{I}^2$.
Let $T$ be the set
$$
T:=\left\{f\in H^0(Z, \mathcal{O}_Z)|\ \forall j,\ 
  f\in H^0(Z, \mathcal{I}_j \mathcal{O}_Z)\setminus \{0\}
\right\}.
$$
For all $n\in \N$ and all $f\in H^0(X, \mathcal{O}_X(n))$, $f_{|Z}\in T$ means that the vanishing locus of $f$ on $X$ contains all the $Q_i$'s and is smooth at each of them.
We conclude by applying Theorem \ref{poonen}.
\end{proof}

\subsection{A vanishing problem} 
As we see further in \ref{keykey}, the statement of Proposition \ref{keytool} expects a vanishing condition on the sheaf cohomology space $H^1(S,\Omega^2(G-X))$, where $S$ is a smooth projective surface and $G,X$ are divisors on $S$.
The point of the present section is to give some criteria on $G$ and $X$ to satisfy such a vanishing condition.

\begin{lem}\label{gensurj}
  Let $S$ be a smooth projective geometrically connected surface over a field $k$, $G$ be an arbitrary divisor on $S$ and $L$ be an ample divisor.
Then, there exists an integer $m$ such that for all $s\geq m$, we have
$$
H^1 (S, \L (G-sL))=H^1(S, \Omega^1(G-sL))=0.
$$
\end{lem}

\begin{proof}
From \cite{H} Corollary III.7.8, the space $H^1(S, \L (G-sL))$ is zero for all $s \gg 0$.
Since $S$ is assumed to be smooth, Serre's duality yields the other equality.
\end{proof}

\begin{lem}\label{complete}
  Let $S$ be a smooth projective geometrically connected surface over a field $k$ which is a complete intersection in a projective space $\P^r_k$ for some $r\geq 3$.
Denote by $H_S$ the hyperplane section on $S$ for this projective embedding.
Let $G$ be a divisor on $S$ such that $G\sim mH_S$ for some integer $m$ and $X\subset S$ be a curve which is a complete intersection in $\P^r$.
Then,
$$
H^1(S, \L (G-X))=H^1(S, \Omega^1(G-X))=0.
$$
\end{lem}

\begin{proof}
Consider the exact sequence of sheaves on $S$
$$
0 \rightarrow \L (G-X) \rightarrow \L (G) \rightarrow i_{\star} \L (G^{\star}) \rightarrow 0,
$$
where $i$ denotes the canonical inclusion map 
$i:X \hookrightarrow S$.
Looking at the long exact sequence in cohomology, we have
\begin{equation}\label{ES}
H^0(S, \L (G)) \rightarrow H^0(X, \L (G^{\star})) \rightarrow H^1 (S, \L (G-X))
\rightarrow H^1 (S, \L (G)).
\end{equation}
Since $G\sim mH_S$, the sheaves $\L (G)$ on $S$ and $\L (G^{\star})$ on $X$ are respectively isomorphic to $\mathcal{O}_S(m)$ and $\mathcal{O}_X(m)$. 
In addition, since $S$ is a complete intersection in $\P^r$, we have 
$
H^1 (S, \L (G))=H^1 (S, \mathcal{O}_S (m))=0$ (see \cite{H} Exercise III.5.5(c)).
Thus (\ref{ES}) together with the above claims yield
\begin{equation}\label{ESbis}
H^0(S, \mathcal{O}_S (m)) \rightarrow H^0(X, \mathcal{O}_X (m)) \rightarrow H^1 (S, \L (G-X))
\rightarrow 0.
\end{equation}
Moreover, from \cite{H} Exercise III.5.5(a), the natural restriction map
$$H^0(\P^r, \mathcal{O}_{\P^r}(m))\rightarrow H^0(X, \mathcal{O}_X (m))$$
is surjective.
Since this map is the composition of
$$
H^0 (\P^r, \mathcal{O}_{\P^r}(m))\rightarrow H^0(S, \mathcal{O}_S (m))
\ \ \textrm{and} \ \ H^0(S, \mathcal{O}_S (m)) \rightarrow H^0(X, \mathcal{O}_X (m)),
$$
the right-hand map above is also surjective.
The exact sequence (\ref{ESbis}) together with the previous assertion yield $H^1(S, \L (G-X))=0$.
  Finally, since $S$ is smooth, Serre's duality entails $H^1(S, \Omega^2 (G-X))=0$.
\end{proof}

\begin{rem}
In Lemma \ref{complete}, the curve $X$ needs not to be a hypersurface section of $S$, one just expects it to be a complete intersection in the ambient space of $S$. For instance, Lemma \ref{complete} can be applied to a line $X$ embedded in $S$.
\end{rem}

\subsection{The key tool}\label{keykey} In the present subsection, we state Proposition \ref{keytool}, which is useful to prove Theorem \ref{1real} (answering Question \ref{qrealsum}b) and then to prove Theorem \ref{Bound} (yielding lower bounds for the minimum distance of duals of functional codes on a surface).

In what follows we always stay in the context presented in \ref{contextcodes}.

\begin{defn}[Support of a codeword]
  In the context of \ref{contextcodes}, given a codeword $c$ in $C_L(S, \Delta, G)$ or its dual, we call \emph{support of $c$} and denote by $\supp (c)$ the set of rational points $\{P_{i_1}, \ldots , P_{i_s}\}$ whose indexes correspond to the nonzero coordinates of $c$.
\end{defn}

\begin{prop}\label{keytool}
In the context of \ref{contextcodes}, let $c\in C_L (S, \Delta , G)^{\bot}$ be a nonzero codeword.
Let $X$ be a reduced curve embedded in $S$, containing the support of $c$ and such that $H^1(S, \Omega^2(G-X))=0$.
Then, there exists a divisor $D$ on $S$ such that
\begin{enumerate}[(i)]
\item\label{key1} $(D,X)$ is sub--$\Delta$--convenient;
\item\label{key2} $c\in C_{\Omega}(S,D,X,G)$.
\end{enumerate}
Moreover, if $X$ is minimal for the property ``$X$ contains $\supp (c)$'' (i.e. any reduced curve $X'\varsubsetneq X$ avoids at least one $P\in \supp (c)$), then
\begin{enumerate}[(i)]
\setcounter{enumi}{2}
\item\label{key3} $w(c)\geq X.(G-K_S-X)$,
\end{enumerate}
where $K_S$ denotes the canonical class on $S$.
\end{prop}

\begin{rem}
  Since $S$ is assumed to be smooth, by Serre's duality, the condition $H^1(S, \Omega^2(G-X))=0$ is equivalent to $H^1 (S, \L (G-X))=0$.
\end{rem}

The following lemma is needed in the proof of Proposition \ref{keytool}.

\begin{lem}\label{mult}
  Let $P$ be a  point of $S$.
Let $C\subset S$ be a smooth curve at $P$ and $X,Y\subset S$ be two other curves such that any two of the curves $C,X,Y$ have no common irreducible component in a neighbourhood of $P$. Then,
$$
m_P(X,Y)\geq \min \{m_P(C,X), m_P(C,Y)\}.
$$ 
\end{lem}

\begin{proof}
  Let $v$ be a local equation of $C$ in a neighbourhood of $P$ and let $u$ be a rational function on $S$ such that $(u,v)$ is a system of local coordinates at $P$.
Let $\phi_X, \phi_Y \in \mathcal{O}_{S,P}$ be respective local equations of $X$ and $Y$ in a neighbourhood of $P$.
Denote by $a_X$ and $a_Y$ the respective $P$--adic valuations of the functions ${\phi_X}_{|C}$ and ${\phi_Y}_{|C}$ on the curve $C$.

Then,
$
m_P(C,X)=\dim \mathcal{O}_{S,P}/(\phi_X,v)= \dim \mathcal{O}_{C,P} / ({\phi_X}_{|C})=a_X,
$
and in the same way, $m_P(C,Y)=a_Y$.
By symmetry, one can assume that $a_X \leq a_Y$. Then, let us prove that $1,u, \ldots, u^{a_X-1}$ are linearly independent in $\mathcal{O}_{S,P} / (\phi_X, \phi_Y)$.
Let $\lambda_0, \ldots , \lambda_{a_X-1} \in {\F}_q$ such that
$$
\lambda_0 + \lambda_1 u +\cdots +\lambda_{a_X-1}u^{a_X-1}=\alpha \phi_X+\beta \phi_Y,
$$
for some $\alpha, \beta \in \mathcal{O}_{S,P}$.
Reduce the above equality modulo $v$. This yields an equality in $\mathcal{O}_{C,P}$ whose right-hand term has $(u)$--adic valuation $\geq a_X$.
Thus, $\lambda_0=\cdots =\lambda_{a_X-1}=0$.
This concludes the proof.
\end{proof}

\begin{proof}[Proof of Proposition \ref{keytool}]
After a suitable reordering of the indexes, one can say that $\supp (c)=\{P_1, \ldots , P_s\}$ for some $s\leq n$.

\medbreak

\noindent \textbf{Step 0.}
Since $S$ is projective, there exists a closed immersion $S\hookrightarrow \P^r$ for some $r\geq 3$.
Let $H_S$ be the corresponding hyperplane section.

\medbreak

\noindent \textbf{Step 1. The curve $C$.}
From Corollary \ref{utile}, there exists a curve $C\subset S$ such that
\begin{enumerate}
\item $C$ is smooth and geometrically connected;
\item $C \nsubseteq X$;
\item $C$ contains $P_1, \ldots, P_s$;
\item $C$ is linearly equivalent to $dH_S$ for some positive integer $d$.
\end{enumerate}
Moreover, Corollary \ref{utile} asserts that $d$ can be chosen to be as large as possible.
Thus, from Lemma \ref{gensurj}, choosing a large enough $d$, we have $H^1 (S, \L (G-C))=0$ and hence 
\begin{enumerate}
\setcounter{enumi}{3}
\item the restriction map
$H^0 (S, \L (G)) \rightarrow H^0(C, \L (G^{\star}))$ 
is surjective.
\end{enumerate}

\medbreak

\noindent \textbf{Step 2. The codeword $c^{\star}$.}\label{step2_1real}
Denote by $F_c$ the divisor on $C$ defined by
$$F_c:=P_{1}+ \cdots +P_{s}\ \in \textrm{Div}(C).$$
The surjectivity of the map $H^0 (S, \mathcal{L}(G)) \rightarrow H^0 (C, \mathcal{L}(G^{\star}))$, induces a na\-tural code map
$\phi: C_L (S, \Delta, G) \rightarrow C_L(C, F_c, G^{\star})$ which is also surjective.
It can be actually regarded as a puncturing map on the functional code on $S$ (see \cite{slmc} 1.9.(II) for a definition).
Therefore, one sees easily that the orthogonal map $\phi^{\bot}:C_L(C, F_c, G^{\star})^{\bot}\rightarrow C_L (S, \Delta, G)^{\bot}$
\begin{enumerate}[(a)]
\item is injective and obtained by extending codewords with $n-s$ zero coordinates on the right;
\item  preserves the Hamming distance;
\item  induces an isomorphism between $C_L(C, F_c, G^{\star})^{\bot}$ and the sub-code of $C_L (S, \Delta,$ $G)^{\bot}$ of codewords having their supports contained in $\{P_1, \ldots , P_s\}$.
\end{enumerate}
Thus, $c$ is in the image of $\phi^{\bot}$.
Denote by $c^{\star}$ the codeword of $C_L (C,F_c,G^{\star})^{\bot}$ such that $\phi^{\bot}(c^{\star})=c$.
It is the punctured codeword $(c_{1}, \ldots , c_{s})$ of $c$ obtained by removing all the zero coordinates.
Obviously, we have $w(c)=w(c^{\star})$.

\medbreak

\noindent \textbf{Step 3. The $\mathbf{1}$--form $\mu$.}
From \cite{sti} Theorem II.2.8, we have $C_L(C,F_c, G^{\star})^{\bot}$ $=C_{\Omega}(C,F_c, G^{\star})$.
Thus, since $c^{\star}\in C_L(C, F_c, G^{\star})^{\bot}$, there exists a $1$--form $\mu \in H^0 (C, \Omega^1(G^{\star}-F_c))$ such that
$$c^{\star}=(\res_{P_1}(\mu), \ldots, \res_{P_s}(\mu)).$$

\medbreak

\noindent \textbf{Step 4. The $\mathbf{2}$--form $\omega$.} As said in \ref{1res}, any rational $2$--form $\nu$ on $S$ with valuation $\geq -1$ along $C$ has a $1$--residue $\res^1_C(\nu)$ on $C$.
This map $\res^1_C$ is actually a surjective sheaf map, yielding the following exact sequence:
$$
0 \rightarrow \Omega^2(G-X) \rightarrow \Omega^2(G-X-C)
\begin{array}[b]{c}
\res^1_C \\
\longrightarrow  
\end{array}
i_{\star}\Omega^1 (G^{\star}-X^{\star})  \rightarrow 0,
$$
where $i$ denotes the canonical inclusion map $i: C\rightarrow S$.
Using the corresponding long exact sequence in cohomology and since, by assumption, $H^1 (S, \Omega^2(G-X))$ is zero, the map
\begin{equation}\label{map}
\res^1_C: H^0(S, \Omega^2 (G-X-C)) \rightarrow H^0(C, \Omega^1 (G^{\star}-X^{\star}))
\end{equation}
is surjective.
Moreover, since $X$ contains the points $P_1, \ldots , P_s$, we have the following divisors inequality on $C$:
$$
0\leq F_c\leq X^{\star}
$$
and hence $H^0(C,\Omega^1(G^{\star}-F_c))\subseteq H^0(C, \Omega^1(G^{\star}-X^{\star}))$.
Thus, $\mu \in H^0(C,$ $\Omega^1(G^{\star}-X^{\star}))$ and, since the map in (\ref{map}) is surjective, there exists a $2$--form $\omega \in H^0(S,$ $\Omega^2(G-X-C))$ such that $\mu=\res^1_C(\omega)$.

\medbreak

\noindent \textbf{Step 5. The divisor $D$.}
The divisor of $\omega$ is of the form 
\begin{equation}\label{dom}
(\omega)=G-X-C+A, \quad \textrm{with}\ A\geq 0.
\end{equation}
Set
\begin{equation}\label{D=}
D :=C-A.
\end{equation}

\medbreak

\noindent \textbf{Step 6. Proof of (\ref{key1}).}
From the definition of sub--$\Delta$--convenience (Definition \ref{sub}), to prove the sub--$\Delta$--convenience of $(D,X)$, we have to prove that $\res^2_{D,P}$ is $\mathcal{O}_{\overline{S},P}$--linear for all $P\in \overline{S}$ and is zero whenever $P\notin \{P_1, \ldots , P_n\}$.
Since the pole locus of $\omega$ is contained in $C\cup X$,
from Proposition \ref{2res} and Remark \ref{machin1}, this map is zero at each $P\notin \overline{C} \cap \overline{X}$.

Moreover, recall that, by Definition of $\res^2_{D,P}$ (see \S \ref{subres}), and from (\ref{D=}) we have $\res^2_{D,P}=\res^2_{C,P}+\res^2_{A,P}$ (by definition, the map depends only on the support $D$, thus it is an addition and not a subtraction).
In addition, since any $\nu\in \overline{\Omega^2(-D-X)}_P$ has no pole along $\supp (A)$, from Proposition \ref{2res}(\ref{2res1}), the map $\res^2_{A,P}$ vanishes on $\overline{\Omega^2(-D-X)}_P$ and hence
\begin{equation}
  \label{res12}
  \res^2_{D,P} \equiv \res^2_{C,P} \ \textrm{on}\ \overline{\Omega^2(-D-X)}_P.
\end{equation}
Thus, let us prove the $\mathcal{O}_{\overline{S},P}$--linearity of $\res^2_{C,P}$ at each $P\in \overline{C}\cap \overline{X}$ and prove that this map is zero if $P\notin \{P_1, \ldots , P_n\}$ (actually, we prove that this map is zero if and only if $P\notin \{P_1, \ldots , P_s\}$, which is stronger).

\medbreak

Let $P \in \overline{C} \cap \overline{X}$ and $f$ be a generator of $\overline{\L(G)}_P$ over $\mathcal{O}_{\overline{S},P}$.
One sees easily that the germ of $f\omega$ generates $\overline{\Omega^2(-D-X)}_P$.
Let $\varphi  \in \mathcal{O}_{\overline{S},P}$. From Proposition \ref{2res}(\ref{2res3}), we have
\begin{equation}\label{fomega}
\res^2_{C,P}(\varphi f \omega)=\res_P (\res^1_C (\varphi f \omega))=\res_P (\varphi_{|C} f_{|C} \mu).
\end{equation}
Moreover, the divisor of $f_{|C}\mu$ satisfies $(f_{|C} \mu)\geq -F_c$ in a neighbourhood of $P$.
Thus, if $P\in \{P_1, \ldots , P_s\}$, then the $1$--form $f_{|C}\mu$ has valuation $-1$ at $P$ and
$$
(\ref{fomega})\ \Rightarrow \ \res^2_{C,P} (\varphi f \omega)=\varphi (P)\res_P(f_{|C} \mu)=\varphi (P) \res^2_{C,P}(f \omega).
$$

Otherwise, if $P \notin \{P_1, \ldots , P_s\}$, then $f_{|C}\mu$ has valuation $\geq 0$ at $P$ and
$$
(\ref{fomega}) \ \Rightarrow \ \res^2_{C,P}(\varphi f \omega)=0.
$$
Thus, $(D,X)$ is sub--$\Delta$--convenient. 
It is actually $\Delta'$--convenient for $\Delta':=P_1+\cdots +P_s$.

\medbreak

\noindent \textbf{Step 7. Proof of (\ref{key2}).}
From (\ref{res12}), we have for all $P\in \overline{S}$,
$\res^2_{D,P}(\omega)=\res^2_{C,P}(\omega)$. Moreover, Proposition \ref{2res}(\ref{2res3}) entails $\res^2_{C,P}(\omega)=\res_P(\res^1_C(\omega))$ $=\res_P(\mu)$.
Thus, 
\begin{equation}\label{c=}
c=\res^2_{D, \Delta}(\omega) \in C_{\Omega}(S,X,D,G).
\end{equation}

\medbreak

\noindent \textbf{Step 8. Proof of (\ref{key3}).}
From now on, assume that $X$ is minimal for the property ``$X$ contains $\supp (c)$''.
First, notice that, from (\ref{dom}) and (\ref{D=}), we have $D\sim G-K_S-X$.
Let us prove that $w(c)\geq X.D$.
For that, we prove that $X$ and $\supp (D)$ have no common irreducible components.
Afterwards, we get inequalities satisfied by all the local contributions $m_P(X,D)$ for all $P\in \overline{S}$ and sum them up to get an inequality satisfied by $X.D$.

\medbreak

\noindent \textit{Sub-step 8.1.} First, let us prove that $X$ and $\supp (D)$ have no common irreducible component. By construction, $C$ is irreducible and not contained in $X$, thus we just have to check that $\supp (A)$ and $X$ have no common irreducible component.
Assume that $A=A'+X_1$, with $A'\geq 0$ and $X_1$ is an irreducible component of $X$.
Set $X':=X\setminus X_1$.
Then, (\ref{dom}) gives $(\omega)=G-C-X'+A'$.
By assumption on the minimality of $X$, the curve $X'$ avoids at least one point in $\{P_1, \ldots, P_s\}$, say $P_1$ after a suitable reordering of the indexes.
Thus, $C$ is the only pole of $\omega$ in a neighbourhood of $P_1$ and, from Proposition \ref{2res}(\ref{2res2}) together with (\ref{res12}), we have $\res^2_{D,P_1}(\omega)=\res^2_{C,P_1}(\omega)=0$.
But, from (\ref{c=}), we have $\res^2_{D,P_1}(\omega)=c_1$ and $c_1\neq 0$ since by assumption, $P_1\in \supp (c)$.
This yields a contradiction.

\medbreak

\noindent \textit{Sub-step 8.2.} Now, let us study the intersection multiplicities $m_P (D,X)$ for all $P \in \overline{S}$.
First, notice that 
\begin{equation}\label{notin}
\forall P \notin \overline{C},\ m_P(X,D)\leq 0.
\end{equation}
Indeed, if $P\notin \overline{C}$, then $m_P(D,X)=m_P(C-A,X)=-m_P(A,X)$ which is negative since $A$ and $X$ are both effective.

To get information on $m_P(X,D)$ for $P\in \overline{C}$, we first study $m_P(C,A-X)$.
From \cite{oim} Lemma 8.8, we have 
$$
\forall P\in \overline{C},\ m_P(C, (\omega)+C)=v_P(\mu),
$$
where $v_P$ denotes the valuation at $P$.
From (\ref{dom}), we get
$$
\begin{array}{crcl}
\forall P\in \overline{C}, &  m_P(C,G-X+A) & = & v_P(\mu)\\
&  m_P(C,A-X) & = & v_P(\mu)-v_P(G^{\star}).
\end{array}
$$

\noindent Afterwards, recall that $(\mu)\geq G^{\star}-P_1-\cdots -P_s$ (see Step 3).
Moreover, since $\mu$ has nonzero residues at the points $P_1, \ldots , P_s$ (its residues at these points are the $s$ first coordinates of $c$ which are assumed to be nonzero), its valuation at these points is equal to $-1$. Consequently, we obtain
\begin{equation}\label{mpcamx}
\forall P\in \overline{C},\ m_P(C,A)-m_P(C,X)\left\{
  \begin{array}{crc}
    \geq & 0 & \textrm{if}\ P\notin \{P_1, \ldots , P_s\}\\
    = & -1 & \textrm{if}\ P\in \{P_1, \ldots , P_s\}
  \end{array}
\right. .
\end{equation}

\noindent Therefore, from Lemma \ref{mult} together with (\ref{mpcamx}), we get
$$
\begin{array}{crcl}
  \forall P \in \overline{C}, & m_P(X,C-A) & \leq & m_P(X,C)-\min \{m_P(C,X), m_P(C,A)\}\\
 & & \leq & \left\{
   \begin{array}{ccc}
     0 & \textrm{if} & P\notin \{P_1, \ldots , P_s\} \\
     m_P(C,X-A) & \textrm{if} & P\in \{P_1, \ldots , P_s\}
   \end{array}.
\right.
\end{array}
$$

\noindent Again from (\ref{mpcamx}), if $P\in \{P_1, \ldots , P_s\}$, then $m_P(C,X-A)=1$.
Thus, if we summarise all the information given by the above inequalities together with (\ref{notin}), we get,
$$
\forall P\in \overline{S},\ m_P(X,D)\leq \left\{
  \begin{array}{ccc}
    0 & \textrm{if} & P\notin \{P_1, \ldots , P_s\}\\
    1 & \textrm{if} &P\in \{P_1, \ldots , P_s\}
  \end{array}
\right. .
$$ 
Finally, summing up all these inequalities gives
$$
X.(G-K_S-X) =X.D \leq s=w(c).
$$
\end{proof}

\section{Differential realisation of the dual of a functional code}\label{secreal}
The first possible application of Proposition \ref{keytool} is the following theorem, which answers Question \ref{qrealsum}b and hence the question raised in the conclusion of \cite{oim}.

\begin{thm}\label{1real}
  Let $S$ be a smooth geometrically connected projective surface over $\F_q$, let $G$ be a divisor on $S$ and $P_1, \ldots , P_n$ be rational points of $S$.
Denote by $\Delta$, the $0$--cycle $\Delta:=P_1 + \cdots , P_n$.
Let $c$ be a codeword of $C_L (S, \Delta, G)^{\bot}$, then there exists a sub--$\Delta$--convenient pair of divisors $(D_a, D_b)$ and a rational $2$--form $\omega\in H^0(S,\Omega^2(G-D_a-D_b))$ such that
$$
c:=\res^2_{D_a,\Delta}(\omega).
$$
Moreover, one of the divisors $D_a,D_b$ can be chosen to be very ample.
\end{thm}

Before proving Theorem \ref{1real}, let us state a straightforward corollary of it yielding a positive answer for Question \ref{qrealsum}.
That is, even if the dual of a functional code on a smooth surface $S$ is not in general a differential code on $S$, \textbf{it is always a sum of differential codes on this surface}.

\begin{rem}
  Actually, using Theorem \ref{1real}, one proves that the dual code $C_L (\Delta, G)^{\bot}$ is a \textbf{union} of differential codes.
\end{rem}

\begin{cor}\label{totalreal}
Under the assumptions of Theorem \ref{1real},
there exists a finite family $(D_a^{(1)}, D_b^{(1)}), \ldots , (D_a^{(r)}, D_b^{(r)})$
of sub--$\Delta$--convenient pairs such that
$$
C_L (\Delta , G)^{\bot}= \sum_{i=1}^r C_{\Omega}(\Delta,D_a^{(i)},D_b^{(i)},G).
$$  
\end{cor}

\begin{proof}[Proof of corollary \ref{totalreal}]
Inclusion $\supseteq$ comes from \cite{oim} Theorem 9.1 and Remark \ref{machin2}.
The reverse inclusion is a consequence of Theorem \ref{1real} together with the finiteness of the dimension of $C_L (S,\Delta,G)^{\bot}$.
\end{proof}

\begin{proof}[Proof of Theorem \ref{1real}]
Since $S$ is assumed to be projective, consider some projective embedding of $S$ and let $H_S$ be the corresponding hyperplane section.

From Corollary \ref{utile}, there exists a smooth geometrically irreducible curve $X$ containing all the support of $c$ and such that $X\sim sH_S$ for some positive integer $s$. Moreover, such a curve $X$ can be chosen with $s$ as large as possible.
Therefore, from Lemma \ref{gensurj}, one can choose $X$ such that $H^1(S, \Omega^2(G-X))=0$.
Set $D_b:=X$ and conclude using Proposition \ref{keytool}.
\end{proof}



\subsection{About Theorem \ref{1real}, some comments and an open question}\label{nonconstr}
Unfortunately, the proof of Theorem \ref{1real} is not constructive.
Indeed, this proof involves the existence of a curve $X$ embedded in $S$ such that $X$ is smooth, is linearly equivalent to $sH_S$ for some integer $s$ and such that $H^1(S, \Omega^1(G-X))=0$.
Poonen's Theorem together with \cite{H} Corollary III.7.8 assert the existence of such a curve provided $s$ is large enough.
However, one cannot estimate or find an upper bound for the lowest possible integer $s$ for which such a curve $X$ exists.

Nevertheless, Theorem \ref{1real} is interesting for theoretical reasons:
it extends to surfaces a well-known result for codes on curves.
Notice that the construction of a differential code on a surface needs a $\Delta$--convenient pair which is not necessary for the construction of a functional code.
Given a functional code $C_L (\Delta, G)$ on a surface $S$, there is no canonical choice of the (sub--)$\Delta$--convenient pair $(D_a, D_b)$
 to construct the code $C_{\Omega}(\Delta, D_a,D_b, G)$.
This lack of canonicity entails the lack of converse inclusion in
$$C_{\Omega}(\Delta, D_a,D_b, G) \subset C_L (\Delta, G)^{\bot}.$$
Basically, Theorem \ref{1real} asserts that $C_L (\Delta, G)^{\bot}$ can be obtained by summing all the differential codes $C_{\Omega} (\Delta, D_a^{(i)}, D_b^{(i)}, G)$ for all possible (sub--)
$\Delta$--convenient pairs $(D_a^{(i)}, D_b^{(i)})$.
Since the dimension of a code is finite, it is sufficient to sum on a finite set of $\Delta$--convenient pairs.
This opens the following question.

\begin{question}\label{number}
  Under the assumptions of Theorem \ref{1real}, what is the minimal number of differential codes whose sum equals $C_L (S,\Delta,G)^{\bot}$? 
\end{question}

\begin{exmp} This number is $1$ when $S$ is the projective plane. Indeed, functional codes on $\P^2$ are Reed--Muller codes (see \cite{slmc} chapter 13) and it is well-known that the dual of a Reed--Muller code is also Reed--Muller (for instance see \cite{handbook} XVI.5.8). Thus, the dual of a functional code on $\P^2$ is also functional and, from \cite{oim} Theorem 9.6, a functional code can be realised as a differential one.
\end{exmp}

\begin{exmp}
It has been proved in \cite{oim} Propositions 10.1 and 10.3 that this number is $2$ when $S$ is the product of two projective lines.
\end{exmp}

\section{Minimum distance of $C_L (S,\Delta, G)^{\bot}$}\label{secmind}

Another application of Proposition \ref{keytool} is to find a lower bound for the minimum distance of a code $C_L (S, \Delta, G)^{\bot}$.
In this section
we stay in the classical context yielding codes on a surface which is described in \ref{contextcodes}. We also introduce a notation.


\begin{nota}
  Denote by $d^{\bot}$ the minimum distance of $C_L (\Delta, G)^{\bot}$.
\end{nota}

\subsection{The naive approach}
The key of the method is to use Proposition \ref{keytool}(\ref{key3}).
Consider a nonzero codeword $c\in C_L (\Delta, G)^{\bot}$.
Let $X$ be a curve containing $\supp (C)$, which is minimal for this property and such that $H^1 (S, \Omega^2 (G-X))=0$. 
Then, Proposition \ref{keytool}(\ref{key3}) asserts that $w(c)\geq X.(G-K_S-X)$.

Basically, one could say that the minimum distance of $C_L (\Delta, G)^{\bot}$ is greater than or equal to the ``minimum of $ X.(G-K_S-X)$ for all $X\subset S$ satisfying the conditions of Proposition \ref{keytool}''.
Unfortunately, it does not make sense since the set of such integers has no lower bound.
Indeed, using Corollary \ref{utile} together with Lemma \ref{gensurj}, one sees that for a large enough integer $r$, there exists a curve $X\sim rH_S$ containing $\supp (c)$, which is minimal for this property (from Corollary \ref{utile}, $X$ can be chosen to be irreducible) and such that $H^1 (S, \Omega^2(G-X))=0$.
Finally, notice that $rH_S.(G-K_S-rH_S)\rightarrow -\infty$ when $r\rightarrow +\infty$.

Thus, the point of the method is to take a minimum in a \textit{good} family of divisor classes, yielding a positive lower bound.

\subsection{The statement}\label{Pdd}
The main result of the present section involves a set of divisor classes which satisfies some properties. The description of these properties is the point of the following definition.

\begin{defn}\label{Q()}
Let $\delta$ be a positive integer.
A set of divisor classes $\mathcal{D}$ on $S$ is said to satisfy the property $\mathcal{Q}(\Delta, G, \delta)$ if it satisfies the following conditions.
  \begin{enumerate}
  \item[($\mathfrak{V}$)] For all $D\in \mathcal{D}$, we have $H^1 (S, \Omega^2(G-D))=0$.
  \item[($\mathfrak{I}$)] For all $\tau$--tuple $P_{i_1}, \ldots , P_{i_{\tau}}$ with $\tau < \delta$, there exists a curve $X \subset S$ whose divisor class is in $\mathcal{D}$ and which contains $P_{i_1}, \ldots , P_{i_{\tau}}$.
Moreover, $X$ is minimal for this property (i.e. any curve $X'\varsubsetneq X$ avoids at least one point of the $\tau$--tuple $P_{i_1}, \ldots , P_{i_{\tau}}$).
  \end{enumerate}
\end{defn}

\begin{nota}\label{delta(D)}
  Given a set of divisor classes $\mathcal{D}$ such that the set $\{D.(G-K_S-D),\ D\in \mathcal{D}\}$ has a smallest element, we denote by $\delta (\mathcal{D})$ the integer
$$
\delta (\mathcal{D}):=\min_{D\in \mathcal{D}} \{D.(G-K_S-D)\}.
$$
\end{nota}

\begin{thm}[Lower bound for $d^{\bot}$]\label{Bound}
In the context described in \ref{contextcodes},
let $\mathcal{D}$ be a set of divisor classes on $S$.
If $\mathcal{D}$ satisfies the property $\mathcal{Q}(\Delta, G, \delta(\mathcal{D}))$, then
$$
d^{\bot}\geq \delta (\mathcal{D})=\min_{D\in \mathcal{D}} \{D.(G-K_S-D)\}.
$$
\end{thm}





\begin{proof}
Let $c$ be a nonzero codeword in $C_L (\Delta, G)^{\bot}$ and assume that $w(c)=\tau<\delta (\mathcal{D})$.
Since $\mathcal{D}$ satisfies $\mathcal{Q}(\Delta, G, \delta(\mathcal{D}))$, there exists a curve $X$ containing $\supp (c)$, which is minimal for this property and whose divisor class is in $\mathcal{D}$.
Moreover, $H^1 (S, \Omega^2 (G-X))=0$.
Therefore, from Proposition \ref{keytool}(\ref{key3}), 
$$
\tau \geq X.(G-K_S-X)\geq \delta (\mathcal{D}),
$$
which yields a contradiction.
\end{proof}

\subsubsection{The arithmetical improvement}
It is possible to improve the bound given by Theorem \ref{Bound} using the maximal number of rational points of an effective divisor whose class is in $\mathcal{D}$. 
For that, let us introduce a notation.

\begin{nota}
  Let $D$ be a divisor class on $S$.
If the corresponding linear system $|D|$ is nonempty, we denote by $\Theta (D)$
the integer
$$
\Theta (D):=\max \{\sharp (\supp (A))(\F_q),\ A\in |D|\}.
$$
\end{nota}

\begin{thm}[Improvement of the lower bound for $d^{\bot}$]\label{impr}
  In the context described in \ref{contextcodes},
let $\mathcal{D}$ be a set of divisor classes on $S$ and $\mathcal{E}$ be a subset of $\mathcal{D}$ such that,
$$
\mathcal{E} \supseteq\{D\in \mathcal{D}: \Theta (D)\geq D.(G-K_S-D)\}
$$
If $\mathcal{Q}(\Delta, G, \delta (\mathcal{E}))$ is satisfied by $\mathcal{D}$, then
$$
d^{\bot}\geq \delta (\mathcal{E}).
$$
\end{thm}

\begin{proof} Let $c$  be a nonzero codeword in $C_L (\Delta, G)^{\bot}$ and assume that $w(c)<\delta (\mathcal{E})$. Since $\mathcal{Q}(\Delta, G, \delta (\mathcal{E}))$ is satisfied by $\mathcal{D}$, there exists a curve $X$ which contains the support of $c$, is minimal for this property and whose divisor class is in $\mathcal{D}$.
Let $D \in \mathcal{D}$ be the divisor class of $X$.
On the one hand, we have, 
$$
\Theta (D)\geq \sharp X(\F_q) \geq w(c).
$$

\noindent On the other hand, from Proposition \ref{keytool}(\ref{key3}), we have 
$$
w(c)\geq D.(G-K_S-D).
$$
Thus, $\Theta (D)\geq D.(G-K_S-D)$ and hence $D\in \mathcal{E}$ and
$w(c)\geq \delta ( \mathcal{E})$.
This yields a contradiction.
\end{proof}

\subsection{How to choose $\mathcal{D}$?}\label{DD}
The most natural choice for $\mathcal{D}$ is $\mathcal{D}=\{H_S, \ldots,$ $aH_S\}$ with $a$ such that $aH(G-K_S-aH)>0$.
From Lemma \ref{complete}, the cohomological vanishing condition of Theorem \ref{Bound} is satisfied by all the elements of $\mathcal{D}$ whenever $S$ is a complete intersection in its ambient space.
Afterwards, one checks whether the interpolation condition is satisfied, if it is not (in particular if the condition of minimality is not satisfied), one can try to add some other divisor classes satisfying the cohomological vanishing condition (for instance see \ref{hyphyp}).

\section{Examples}\label{ex1}

In this section we treat some examples of surfaces and obtain lower bounds or exact estimates of the dual minimum distance of a code.
The difficult part to apply Theorems \ref{Bound} and Theorem \ref{impr} is first to choose a good $\mathcal{D}$ and then to compute $\delta (D)$.
It becomes easier when the Picard number of the surface (that is the rank of its Neron-Severi group) is small.

Most of the examples we give correspond to surfaces with Picard Number $1$.
Some examples of surfaces having a larger Picard number are treater and it turns out that surfaces with Picard number $1$ yield the better duals of functional codes.
Such a remark should be related with the works of Zarzar in \cite{zarzar} who noticed that surfaces with a small Picard number could yield good functional codes.

\subsection{The projective plane}\label{P2}
On $\P^2$, the functional codes are Reed--Muller codes and it is well-known that the dual Reed--Muller code is also a Reed--Muller code (\cite{delsarte} Theorem 2.2.1).
The minimum distance of a $q$--ary Reed--Muller code is well-known (see \cite{delsarte} Theorem 2.6.1).
Therefore, the point of the present subsection is not to give any new result but to compare the bound given by Theorem \ref{Bound} to the exact value of the minimum distance in order to check the efficiency of Theorem \ref{Bound}.

\subsubsection{Context}
Let $H$ be a line on $\P^2$ and $m$ be a nonnegative integer.
Assume that $G:=mH$ and $\Delta:=P_1+\cdots+P_{q^2}$ is the sum of all rational points of the affine chart $\P^2\setminus H$.


\subsubsection{The known results on Reed--Muller codes}
From \cite{delsarte} Theorem 2.2.1, we have $C_L(\Delta,mH)^{\bot}=C_L(\Delta,(2q-3-m)H)$.
Moreover, \cite{delsarte} Theorem 2.6.1 asserts that the minimum distance $d^{\bot}$ of $C_L(\Delta, mH)^{\bot}$ is
\begin{equation}\label{dRM}
d^{\bot}=\left\{
  \begin{array}{ccc}
    m+2 & \textrm{if} & m\leq q-3 \\
    q(m+3-q) & \textrm{if} & m\geq q-2
  \end{array}
\right.
\end{equation}

\subsubsection{Our bounds}
First, recall that $K_{\P^2}\sim -3H$. Therefore,
$$aH.(G-K-aH)=a(m+3-a)H^2=a(m+3-a)$$
and this integer is positive for $1 \leq a \leq m+2$.
Then, set 
$
\mathcal{D}:=\{H,2H,\ldots,$ $(m+2)H\}.
$
This yields $\delta (\mathcal{D})=m+2$ (see notation \ref{delta(D)}).
Thus, we have to prove that it satisfies the property $\mathcal{Q}(\Delta,G, m+2)$.
From Lemma \ref{complete}, this set of divisor classes satisfies the cohomological vanishing condition ($\mathfrak{V}$) (see Definition \ref{Q()}).
Moreover, for all $l\leq m+1$ any $l$--tuple of rational points of $\P^2$ is contained in a curve of degree $\leq m+2$ and one of them is minimal for this property.
Thus, $\mathcal{Q}(\Delta, G, m+2)$ is satisfied and from Theorem \ref{Bound}, we have
\begin{equation}\label{d1}
\forall m,\ d^{\bot}\geq m+2.
\end{equation}

Now, let us improve the result using Theorem \ref{impr}.
First, notice that any configuration of rational points of an affine chart of $\P^2$ is contained in a curve of degree at most $q$.
Therefore, if $m+2\geq q$, one can set
$
\mathcal{D}:=\{H, \ldots , qH\}
$ and
the property $\mathcal{Q}(\Delta, G, s)$ is true for all $s$.
From \cite{lettre}, we have $\Theta (aH)=aq$.
Thus, if $m\geq q-2$, then
$$
\Theta (aH)<a(m+3-a)\quad \textrm{for all}\quad a<m+3-q.
$$
Thus, set
$
\mathcal{E}:=\{(m+3-q)H, \ldots , qH\}.
$
Finally, since $\mathcal{Q}(\Delta, G, s)$ is satisfied by $\mathcal{D}$ for all $s$, it is in particular satisfied for $s=\delta (\mathcal{E})$.
Consequently, from Theorem \ref{impr}, we get
\begin{equation}\label{d2}
\forall m\geq q-2,\ d^{\bot}\geq \delta (\mathcal{E})=q(m+3-q).
\end{equation}

By comparing (\ref{dRM}) with (\ref{d1}) and (\ref{d2}), we see that Theorems \ref{Bound} and \ref{impr} yield exactly the minimum distance of a Reed--Muller code.

\begin{rem}
  By the very same manner one can recover the minimum distance of projective Reed--Muller codes.
\end{rem}

\subsection{Quadric surfaces in $\P^3$}\label{qua}
We study the code $C_L (\Delta, G)^{\bot}$ when $S$ is a smooth quadric in $\P^3$.
Recall that there are two isomorphism classes of smooth quadrics in $\P^3$ called respectively \textit{elliptic} and \textit{hyperbolic}.
A hyperbolic quadric contains two families of lines defined over $\F_q$ and its Picard group is free of rank $2$ and generated by the respective classes $E$ and $F$ of these two families of lines.
An elliptic quadric does not contain lines defined over $\F_q$ and its Picard group is free of rank $1$ and generated by $H_S$.
We treat separately these two cases ($S$ is hyperbolic and $S$ is elliptic).

\subsubsection{Context}\label{contqua}
Let $S$ be a smooth quadric surface in $\P^3$. 
Let $H_S$ be the scheme-theoretic intersection between $S$ and its tangent plane at some rational point.
Let $G$ be $G := mH_S$ for some $m>0$
and $\Delta$ be the sum of all the rational points lying in the affine chart $S\setminus H_S$.
The number of these points (and hence the length of the codes) is $q^2$ and we denote them by $P_1, \ldots , P_{q^2}$.

For all $1\leq m\leq q-1$.
The dimension of the code $C_L (\Delta, G)$ is equal to
\begin{equation}\label{dimdim}
\dim C_L (\Delta, G)=\dim \Gamma (S, \mathcal{O}_S(m))=\left(
  \begin{array}{cc}
    m+3\\
    3
  \end{array}
\right)-
\left(
  \begin{array}{cc}
    m+1\\
    3
  \end{array}
\right)=(m+1)^2.
\end{equation}

\begin{rem}\label{irrel}
  For $m\geq q-1$ we get $C_L (\Delta, G)=\F_q^{q^2}$. Therefore, cases when $m\geq q-1$ are irrelevant.
In what follows, we always assume that $m\leq q-2$.
\end{rem}

Finally, recall that, from \cite{H} Example II.8.20.3,
\begin{equation}
  \label{canonical}
  K_S\sim -2H_S.
\end{equation}

\subsubsection{Hyperbolic quadrics}\label{hyphyp}
If $S$ is a hyperbolic quadric, then, as said before, its Picard group is generated by two lines denoted by $E$ and $F$. Moreover, $E+F\sim H_S$.
As proposed in \ref{DD}, one can set
$
\mathcal{D}:=\{H_S,\ldots, (m+1)H_S\}
$.
This yields $\delta (\mathcal{D})= 2m+2$.
Unfortunately, since $m\leq q-2$, and since $S$ contains rational lines, there are collinear $(m+2)$--tuples of points in $\{P_1, \ldots , P_{q^2}\}$.
For such a $(m+2)$--tuple, there exists hypersurface sections of $S$ of degree $\leq m+1$ containing these points but none of them is minimal for this property since such a curve contains the line containing the $(m+2)$--tuple together with another irreducible component.

Therefore, to apply Theorem \ref{Bound}, we have to add other divisor classes to $\mathcal{D}$. Therefore, set
$$
\mathcal{D}:=\{E,F,H_S, \ldots , (m+1)H_S\}.
$$
We have $\delta (\mathcal{D})=m+2$ and for such a $ \mathcal{D}$, the property $\mathcal{Q}(\Delta, G, \delta (\mathcal{D}))$ satisfied. 
Indeed, since $E,F$ and hypersurface sections of $S$ are complete intersections in $\P^3$, from Lemma \ref{complete}, the cohomological vanishing condition is satisfied.
The proof that the interpolating condition ($\mathfrak{I}$) (see Definition \ref{Q()}) is also satisfied is left to the reader.
Finally, we have the following result.

\begin{prop}\label{hyper}
The minimum distance $d^{\bot}$ of $C_L (\Delta, G)^{\bot}$ satisfies 
$$d^{\bot} = (\mathcal{D})=E.((m+2)H_S-E)=m+2.$$
\end{prop}

\begin{proof}
The inequality $\geq$ is a consequence of Theorem \ref{Bound}.
For the converse inequality, consider a rational line $L$ contained in $S$.
After a suitable change of coordinates, one can assume that $P_1, \ldots , P_q\in L$.
Therefore, the punctured code $C^{\star}$ obtained from $C_L (\Delta, G)$ by keeping only the $q$ first coordinates, can be regarded as a code on $L$, that is a a Reed--Solomon code of length $q$ and dimension $m+1$. 
From well--known results on Reed--Solomon codes, its dual has minimum distance $m+2$ and a minimum weight codeword $c\in {C^{\star}}^{\bot}$ extended by zero coordinates yields a codeword in $C_L (\Delta, G)^{\bot}$ with the same weight.
\end{proof}

\subsubsection{Elliptic quadrics}
If $S$ is an elliptic quadric. This time, since it does not contain rational lines, the set
$$
\mathcal{D}:=\{H_S, \ldots , (m+1)H_S\}
$$
satisfies $\mathcal{Q}(\Delta, G, \delta (\mathcal{D}))$.
Indeed, from (\ref{dimdim}), $\dim \Gamma (S, \mathcal{O}_S(m+1))=(m+2)^2$ which is $>2m+1$. Therefore, any $(2m+1)$--tuple of points in $\supp (\Delta)$ is contained in some curve $C\sim aH_S$ with $a\leq m+1$.
Moreover, since $H_S$ generates the Picard Group of $S$, for some $a\leq m+1$ there exists such a curve $C$ which is minimal for this property.
This yields the following bound.

\begin{prop}\label{ellip1}
  The minimum distance  $d^{\bot}$ of $C_L (\Delta, G)^{\bot}$ satisfies
$$
d^{\bot}\geq 2m+2.
$$
\end{prop}

Moreover, using Theorem \ref{impr}, it is possible to improve efficiently this bound for some values of $m$.
For that, we have to estimate $\Theta (mH_S)$ for all $m\leq q-2$ or find an upper bound for it.
For that we use what we know about the Picard group of $S$ together with the bound proved by Aubry and Perret in \cite{aubryperret} Corollary 3.

Let us give some upper bound for $\Theta (mH_S)$ for some particular values of $m$.
\begin{itemize}
\item $\Theta (H_S)=q+1$, indeed it is the maximal number of rational points of a plane section of $S$ which is a plane conic.
\item $\Theta (2H_S)\leq \max (2(q+1), q+1+\lfloor 2\sqrt{q} \rfloor)=2q+2$. Indeed, a quadric section of $S$ is either irreducible and has arithmetical genus $1$ or reducible. If it is reducible, since the Picard group is generated by $H_S$, it is the union of two curves both linearly equivalent to $H_S$ and hence the union of to plane sections (i.e. of two plane conics).
\item $\Theta (3H_S)\leq \max(3(q+1), q+1+4\lfloor 2\sqrt{q} \rfloor)$.
\item etc...
\end{itemize}

\subsubsection{Numerical application}
To conclude this section on quadrics, let us compare the parameters $[n,k,d]$ of the code $C_L (\Delta, G)^{\bot}$ obtained for particular values of $q$.
The following results are obtained using Propositions \ref{hyper}, \ref{ellip1} and the previous estimates for $\Theta (mH_S)$.

\medbreak

\noindent \textbf{Comparison with Best known codes.} In what follows, the minimum distances of the studied codes are compared with the best known minimum distances  for given length and dimension appearing in \url{www.codetables.de} \cite{codetables} and \url{http://mint.sbg.ac.at} \cite{minT}.
These best known minimum distances appear in the right hand column of each array.

\medbreak

\noindent \textbf{For} $\mathbf{q=4}$.
$$
\begin{array}{|c|c|c|c|c|c|}
\hline
  \multirow{4}{*}{m} & \multirow{4}{*}{\textrm{Length}} & \multirow{4}{*}{\textrm{Dimension}} & \multicolumn{2}{|c|}{\textrm{Minimum}} &
    \\
 & & & \multicolumn{2}{|c|}{\textrm{Distance}} & \textrm{Best Known} \\
\cline{4-5}
 & & & \textrm{Hyperbolic} & \ \ \  \textrm{Elliptic}\ \ \  & \textrm{Distance}  \\
 & & & \textrm{Quadric} & \textrm{Quadric} & \\
\hline
1 & 16 & 12 & 3 & \geq 4 & 4 \\
\hline
2 & 16 & 7 & 4 & \geq 6 & 8 \\
\hline
\end{array}
$$

\bigbreak

\noindent \textbf{For} $\mathbf{q=8}$.
$$
\begin{array}{|c|c|c|c|rlc|c|}
\hline
  \multirow{4}{*}{m} & \multirow{4}{*}{\textrm{Length}} & \multirow{4}{*}{\textrm{Dimension}} & \multicolumn{4}{|c|}{\textrm{Minimum}} &
    \\
 & & & \multicolumn{4}{|c|}{\textrm{Distance}} & \textrm{Best Known} \\
\cline{4-7}
 & & & \textrm{Hyperbolic}  & \multicolumn{3}{|c|}{\textrm{Elliptic}} & \textrm{Distance}  \\
 & & & \textrm{Quadric} & \multicolumn{3}{|c|}{\textrm{Quadric}} & \\
\hline
1 & 64 & 60 & 3 & \geq &4& & 4 \\
\hline
2 & 64 & 55 & 4 & \geq & 6& & 6 \\
\hline
3 & 64 & 48 & 5 & \geq & 8& & 11 \\
\hline
4 & 64 & 39 & 6 & \geq & 16& (\ref{a}) & 16 \\
\hline
5 & 64 & 28 & 7 & &24 & (\ref{b}) & 24 \\
\hline
6 & 64 & 15 & 8 & \geq& 32& (\ref{c}) & 38 \\
\hline
\end{array}
$$
 
\bigbreak

\begin{enumerate}[(a)]
\item\label{a} Take $\mathcal{D}:=\{H_S, \ldots , 4H_S\}$.
Since $\Theta (H_S)\leq 9$ and $H_S. (4H_S-K_S-H_S)=10$, we can choose $\mathcal{E}:=\{2H_S, 3H_S, 4H_S\}$.
We have $\delta (\mathcal{E})=16$ and $\mathcal{Q}(\Delta, G, 16)$ is satisfied by $\mathcal{D}$ since $\dim \Gamma (S, \mathcal{O}_S (4))=25> 16$.
Then, apply Theorem \ref{impr}. 
\item\label{b} Take $\mathcal{D}:=\{H_S, \ldots, 4H_S\}$.
We have $\Theta (2H_S)\leq 18$ and $2H_S.(5H_S-K_S-2H_S)=20>18$.
Take $\mathcal{E}:=\{3H_S, 4H_S\}$ and apply Theorem \ref{impr}.
\item\label{c} Take $\mathcal{D}:=\{H_S, \ldots , 5H_S\}$ and $\mathcal{E}:=\{4H_S\}$.
\end{enumerate}

\bigbreak

\noindent {\bf Note on the $[64, 28, 24]$ code over $\F_8$.}
When this article has been submitted, the best $[64,28]$ code over $\F_8$ on Codetables \cite{codetables} and MinT \cite{minT} had minimum distance $23$.
However, in \cite{DuuChe} Table IIA, Duursma and Chen, assert the existence of a $[64, 28, 24]$ code from the Suzuki curve, without providing further details.
After communicating our results to Markus Grassl (from Codetables), he re-constructed our code using \emph{Construction X}, based on two cyclic codes deriving from ours. By this way, he proved by computer that the exact minimum distance is $24$.
More recently, Iwan Duursma communicated to Markus Grassl a {\sc Magma} script to generate their Suzuki code. He also explained how to deduce the minimum distance of their code. The result comes from a {\sc Magma} computation (\cite{DuuChe} \S III.B. for $k=11$) and a duality argument (\cite{bouquindiego} page 26).
Taking these contributions into account, Codetables has been updated.

\medbreak

\noindent \textbf{For} $\mathbf{q=16.}$ We do not apply the result for all the possible values of $m\leq q-2=14$ since the array would be too long.
Let us only give some of them yielding some relevant codes over the elliptic quadric.

$$
\begin{array}{|c|c|c|c|rlc|c|}
\hline
  \multirow{4}{*}{m} & \multirow{4}{*}{\textrm{Length}} & \multirow{4}{*}{\textrm{Dimension}} & \multicolumn{4}{|c|}{\textrm{Minimum}} &
    \\
 & & & \multicolumn{4}{|c|}{\textrm{Distance}} & \textrm{Best Known} \\
\cline{4-7}
 & & &  \textrm{Hyperbolic} & \multicolumn{3}{|c|}{\textrm{Elliptic}} & \textrm{Distance}  \\
 & & & \textrm{Quadric} & \multicolumn{3}{|c|}{\textrm{Quadric}} & \\
\hline
8 & 256 & 175 & 10 & \geq & 32& (\ref{I1}) & 46 \\
\hline
9 & 256 & 156 & 11 & \geq & 48& (\ref{I2}) & 59 \\
\hline
10 & 256 & 135 & 12 & \geq &64& (\ref{I3}) & 74 \\
\hline
\end{array}
$$

\bigbreak

\begin{enumerate}[(a)]
\item\label{I1} Take $\mathcal{D}:=\{H_S, \ldots, 8H_S\}$. Since $\Theta (H_S)\leq 17$ and $H_S. (8H_S-K_S-H_S)=18>17$, one can take $\mathcal{E}:=\{2H_S, \ldots , 8H_S\}$.
\item\label{I2} Take $\mathcal{D}:=\{H_S, \ldots, 8H_S\}$.
Since $\Theta (2H_S)\leq 34$ and $2H_S (9H_S-K_S-2H_S)=36>34$, one can take $\mathcal{E}:=\{3H_S, \ldots , 8H_S\}$.
\item \label{I3}  Take $\mathcal{D}:=\{H_S, \ldots, 8H_S\}$.
Since $\Theta (3H_S)\leq 51$ and $3H_S (10H_S-K_S-3H_S)=54>51$, one can take $\mathcal{E}:=\{4H_S, \ldots , 8H_S\}$.
\end{enumerate}


\subsection{Cubic surfaces in $\P^3$}
The classification of smooth cubic surfaces is far from being as simple as that of smooth quadrics (see \cite{swd}).
However, in terms of codes, it is sufficient to separate them into two \textit{sets}, the cubics which contain rational lines and those which do not.
As in the case of quadrics, we see that the best codes are given by cubics which do not contain rational lines.

\subsubsection{Context}\label{contcub}
The context is almost the same as that of \ref{contqua}.
Let $S$ be smooth cubic surface in $\P^3$, let $G$ be of the form  $mH_S$ where $H_S$ is a hyperplane section and $\Delta$ be the sum of rational points of $S$ lying out of the support of $H_S$.
For the same reason as in Remark \ref{irrel}, we assume that $m\leq q-2$.

If $m\leq q-2$, then the dimension of $C_L(\Delta, G)$  equals that of $\Gamma (S, \mathcal{O}_S(m))$ which is 
\begin{equation}\label{dimcub}
\dim C_L (\Delta, G)=\left(
  \begin{array}{c}
    m+3\\3
  \end{array}
\right)
-
\left(
  \begin{array}{c}
    m\\3
  \end{array}
\right)=\frac{3m^2+3m+2}{2}\cdot
\end{equation}

\begin{rem}
  There exists cubic surfaces which does not contain any rational line for instance, explicit examples are given in \cite{zarzar} and \cite{agctvoloch}.
Moreover, it is proved in \cite{kollar} that such surfaces have Picard number $1$.
\end{rem}

\subsubsection{Cubics containing rational lines}

\begin{prop}\label{withlines}
  In the context described in \ref{contcub}, if $S$ contains rational lines, then the minimum distance $d^{\bot}$ of $C_L (\Delta, G)^{\bot}$ satisfies 
$$
d^{\bot}=m+2.
$$
\end{prop}

\begin{proof}
Let $L_1, \ldots , L_r$ be all the rational lines contained in $S$.
Set $\mathcal{D}:=\{L_1, \ldots ,$ $L_r, H_S, \ldots , mH_S\}$.
A computation gives $\delta (\mathcal{D})=m+2$ (the minimum is reached by the lines $L_i$).
By the same manner as Proposition \ref{hyper}, the inequality $d^{\bot}\geq m+2$ is given by Theorem \ref{Bound} and the equality is obtained using the very same argument as that of Proposition \ref{hyper}.  
\end{proof}

\subsubsection{Cubics containing no rational lines}
As for elliptic quadrics, we first give a general lower bound based on Theorem \ref{Bound} and then an improvement of it based on Theorem \ref{impr}.

\begin{prop}\label{nolines}
  In the context described in \ref{contcub}, if $S$ does not contain any rational line, then the minimum distance $d^{\bot}$ of $C_L (\Delta, G)^{\bot}$ satisfies
$$
d^{\bot}\geq 3m.
$$
\end{prop}

\begin{proof}
  Set $ \mathcal{D}:=\{H_S, \ldots , mH_S\}$.
We get $\delta (\mathcal{D})= 3m$.
Using (\ref{dimcub}), one proves easily that $\dim \Gamma (S, \mathcal{O}_S(m))\geq 3m-1$ for all $m$ and hence, for all $r<3m$, any $r$--tuple of rational points of $S$ is interpolable by some surface section of $S$ of degree $\leq m$ and one of them is minimal for this property.
Thus, the result is a consequence of Theorem \ref{Bound}.
\end{proof}

  It is easy to compare Propositions \ref{withlines} and \ref{nolines} and see that, as in the case of quadrics, cubics containing no rational lines yield much better codes.
In what follows, we treat numerical examples based on a cubic with no rational lines and see how to use Proposition \ref{nolines} and how to improve its result in some situations using Theorem \ref{impr}.

\subsubsection{Numerical application}
In \cite{zarzar}, the author looked at surfaces with Picard number $1$ to get good functional codes $C_L (\Delta, G)$.
For that, he noticed that in the classification of cubic surfaces up to isomorphism given by Swinnerton--Dyer in \cite{swd} table 1, there exists cubic surfaces which do not contain rational lines and have $q^2+2q+1$ rational points.
Some explicit examples of such surfaces are given in \cite{zarzar} and \cite{agctvoloch}.
The following array gives the parameters of codes arising from such a surface over $\F_9$.

\medbreak

$$
\begin{array}{|c|c|c|rlc|c|}
\hline
  \multirow{2}{*}{m} & \multirow{2}{*}{\textrm{Length}} & \multirow{2}{*}{\rm Dimension} & \multicolumn{3}{|c|}{\textrm{Minimum}} & \textrm{Best Known} \\ 
 & & & \multicolumn{3}{|c|}{\textrm{Distance}} &  \textrm{Distance} \\
\hline
2 & 100 & 90 & \geq & 6 & & 6 \\
\hline
3 & 100 & 81 & \geq & 9 & & 10 \\
\hline
4 & 100 & 69 & \geq & 12 & & 16 \\
\hline
6 & 100 & 36 & \geq & 30 & (\star) & 40 \\
\hline
\end{array}
$$

\bigbreak

The box marked with a ($\star$) corresponds to one where one can apply the improvement given by Theorem \ref{impr}.
Indeed, $\Theta (H_S)\leq 9+1+2\sqrt{9}=16$.

In the same way, using such an improvement, over $\F_8$, with $m=5$ one can get a $[81, 35, 24]$--code.

\subsection{Comment and conclusion}
Looking at the results given in \cite{aubryquad} and \cite{fred2} it is clear that codes of the form $C_L(\Delta, H_S)$ and $C_L(\Delta, 2H_S)$ on elliptic quadrics are much better than codes on hyperbolic ones.
Such a fact holds probably for codes $C_L (\Delta, G)$ on a quadric for more general divisors $G$.

The previous result shows that elliptic quadrics yield also better codes of the form $C_L (\Delta, G)^{\bot}$ that hyperbolic ones.
In both cases, the weakness of hyperbolic quadrics comes from the numerous rational lines they contain.
This fact can be related to the work of Zarzar who noticed in \cite{zarzar} that one could find good codes of the form $C_L (\Delta, G)$ on surfaces having a small Picard Number.
This is well illustrated by quadrics, since hyperbolic quadrics have Picard number $2$ and elliptic ones have Picard number $1$.

Moreover, the principle asserting that surfaces with a small Picard number yield good codes seems to hold for codes of the form $C_L (\Delta, G)^{\bot}$.
At least, the above examples on quadrics and cubic surfaces encourage to look in this direction.
Another explanation makes feel that such surfaces should give good codes:
basically, if the Picard number is small, the set of divisor classes $\mathcal{D}$ of Theorem \ref{Bound} may be small and yield a larger candidate $\delta (\mathcal{D})$ for a lower bound of the minimum distance of $C_L (\Delta, G)^{\bot}$.

Finally, surfaces with small Picard number are twice interesting for coding theory, either for functional codes or for their duals.

\section*{Acknowledgements}
The author wishes to thank Tom H\o holdt and Felipe Voloch for inspiring discussions and Marc Perret for many relevant suggestions on this article.
A computer aided analysis on one of our codes has been done by Markus Grassl who should by the way be congratulated his involvement in Codetables.
Finally, the author expresses his gratitude to Iwan Duursma for some very interesting conversations.

\bibliographystyle{abbrv}
\bibliography{biblio}

\end{document}